\documentclass[a4paper,11pt]{article}

\usepackage[utf8]{inputenc}
\usepackage[english]{babel}
\usepackage{fullpage}
\usepackage{xspace}
\usepackage{amsmath}
\usepackage{amsthm}
\usepackage{amssymb}
\usepackage[bookmarks=true,
            bookmarksnumbered=true,
            pagebackref=true,
            colorlinks=true,
            linkcolor=blue,
            citecolor=red]{hyperref}

\title{The strength of Ramsey Theorem for coloring relatively large sets}

\author{Lorenzo Carlucci \\
  Dipartimento di Informatica, University of Rome I ``La Sapienza''\\
  Rome, Italy\\
  \texttt{carlucci@di.uniroma1.it} \qquad
  \and
  Konrad Zdanowski\\
  University of Cardinal Stefan Wyszy\'nski\\
Warsaw, Poland\\
  \texttt{k.zdanowski@uksw.edu.pl}
}

\date{\today{}}

\hypersetup{
  pdfauthor={Lorenzo Carlucci, Konrad Zdanowski}
  pdftitle ={The strength of Ramsey Theorem for coloring relatively large sets}
}

\newtheorem{definition}{Definition}
\newtheorem{proposition}{Proposition}
\newtheorem{theorem}{Theorem}
\newtheorem{corollary}{Corollary}
\newtheorem{lemma}{Lemma}

\newtheorem{open.problem}{Open Problem}

\newtheorem*{theorem*}{Theorem}
\newtheorem*{corollary*}{Corollary}
\newtheorem*{proposition*}{Proposition*}
\newtheorem*{lemma*}{Lemma}
\newtheorem*{fact*}{Fact}
\newtheorem*{claim*}{Claim}
\newtheorem*{open.problem*}{Open Problem}
\newtheorem*{remark*}{Remark}
\newtheorem*{example*}{Example}
\newtheorem*{exercise*}{Exercise}

\newcommand{\set}[1]{{\{#1\}}}

\newcommand{\pair}[1]{{\langle #1 \rangle}}
\newcommand{\imp}{\Rightarrow}
\newcommand{\pimp}{\Leftrightarrow}

\newcommand{\stops}{{\downarrow}}
\newcommand{\bN}{{\bf N}}

\newcommand{\RT}{\textsc{RT}}
\newcommand{\KM}{\mathsf{KM}}

\newcommand{\vp}{\varphi}

\newcommand{\PA}{\textsc{PA}}
\newcommand{\RCA}{\textsc{RCA}}
\newcommand{\ACA}{\textsc{ACA}}
\newcommand{\ATR}{\textsc{ATR}}

\newcommand{\TI}{\mathbf{TI}}

\newcommand{\card}{\textrm{card}}

\newcommand{\Tr}{\ensuremath{\mathsf{Tr}}\xspace{}}

\newcommand{\isfunc}[3]{{#1}\colon{#2}\rightarrow{#3}}
\newcommand{\goesto}{\rightarrow}

\newcommand{\Nat}{\mathbf{N}}

\begin{document} 

\maketitle

\begin{abstract}
We characterize the computational content and the proof-theoretic strength 
of a Ramsey-type theorem for bi-colorings of so-called {\em exactly large} sets. 
An {\it exactly large} set is a set $X\subset\Nat$ such that $\card(X)=\min(X)+1$. 
The theorem we analyze is as follows. For every infinite subset $M$ of $\Nat$, for every coloring $C$
of the exactly large subsets of $M$ in two colors, there exists and infinite
subset $L$ of $M$ such that $C$ is constant on all exactly large subsets of $L$.
This theorem is essentially due to Pudl\`ak and R\"odl and independently to Farmaki. 
We prove that --- over Computable Mathematics --- this theorem is equivalent to 
closure under the $\omega$ Turing jump (i.e., under arithmetical truth). Natural 
combinatorial theorems at this level of complexity are rare. Our results give a complete
characterization of the theorem from the point of view of Computable Mathematics
and of the Proof Theory of Arithmetic. This nicely extends the current knowledge
about the strength of Ramsey Theorem. We also show that analogous results
hold for a related principle based on the Regressive Ramsey Theorem. 
In addition we give a further characterization in terms of truth predicates over Peano Arithmetic.
We conjecture that analogous results hold for larger ordinals. 
\end{abstract}

\section{Introduction}
A finite set $X\subseteq \Nat$ is {\em large} if $\card(X) > \min(X)$. A finite set $X\subseteq \Nat$
is {\em exactly large} if $\card(X) = \min(X)+1$. The concept of large set was introduced by Paris and Harrington \cite{Par-Har:77}
and is the key ingredient of the famous Paris-Harrington principle, also known as the Large Ramsey Theorem. The latter is the first
example of a natural theorem of finite combinatorics that is unprovable in Peano Arithmetic. We are 
interested in the following extension of the Infinite Ramsey Theorem to bicolorings of exactly large sets.
\begin{theorem}[Pudl\`ak-R\"odl \cite{Pud-Rod:82} and Farmaki \cite{Far:ta,Far-Neg:08}]\label{thm:rtomega}
For every infinite subset $M$ of $\Nat$, 
for every coloring $C$ of the exactly large 
subsets of $\Nat$ in two colors, there exists an infinite 
set $L\subseteq M$ such that every exactly large subset
of $L$ gets the same color by $C$.
\end{theorem}
We refer to the statement of the above Theorem as $\RT(!\omega)$ (the `!' is mnemonic for `exactly', 
while the reason for the use of `$\omega$' is that large sets are also known as `$\omega$-large sets'). 
By an {\it instance} of $\RT(!\omega)$ we indicate a pair $(M,C)$ of the appropriate type. 
Theorem \ref{thm:rtomega} --- with slightly different formulations --- has been essentially proved by 
Pudl\`ak and R\"odl \cite{Pud-Rod:82} and independently by Farmaki \cite{Far:ta,Far-Neg:08}. 
Pudl\`ak and R\"odl's version is stated in terms of `uniform families'. Farmaki's version is 
in terms of Schreier families. Schreier families, originally defined in \cite{Sch:30}, 
play an important role in the theory of Banach spaces. 
The notion has been generalized to countable ordinals in \cite{Als-Ode:88,Als-Arg:92,Tom:96}.
In fact, both \cite{Pud-Rod:82} and \cite{Far:ta} prove a generalization of the above theorem 
to any countable ordinal (see {\it infra\/} for more details). 
As observed in \cite{Far-Neg:08}, Schreier families turn out to essentially coincide with 
the concept of exactly large set. The classical Schreier family is defined as follows
$$ \{ s = \{ n_1,\dots,n_k\} \subseteq \Nat \;:\; n_1 < \dots < n_k \mbox{ and } n_1 \geq k\},$$
while the `thin Schreier family' $\mathcal{A}_\omega$ is defined by imposing $n_1=k$ (see, e.g., \cite{Far-Neg:08}). 
Thus, the Schreier family $\mathcal{A}_\omega$ is just an inessential variant of the family of exactly large subsets of $\Nat$. 

In the present paper we investigate the computational and proof-theoretical content of $\RT(!\omega)$. 
That is, we characterize the complexity of homogeneous sets witnessing the truth of computable instances 
of $\RT(!\omega)$ and we characterize the theorem in terms of formal systems of arithmetic (in 
the spirit of Reverse Mathematics \cite{Sim}). 

In particular, we show that there are computable colorings of the exactly large subsets of $\Nat$
in two colors all of whose homogeneous sets compute the Turing degree $0^{(\omega)}$.
The degree $0^{(\omega)}$ is well-known to be the degree of arithmetical truth, i.e., 
of the first-order theory of the structure $(\Nat,+,\times)$ (see, e.g., 
\cite{Rog}). We show also a reversal of these results by proving that a solution to an instance of $\RT(!\omega)$
can always be found within the $\omega$th Turing jump of the instance. 

Our proofs are such that we obtain as corollaries of the just described computability results the
following proof-theoretical results. First, we show that --- over Computable Mathematics --- 
$\RT(!\omega)$ implies closure under the $\omega$-jump
(or, equivalenty, under arithmetic truth): in terms of Reverse 
Mathematics, we prove that $\RT(!\omega)$ implies --- over $\RCA_0$ --- the axiom stating the existence 
of the $\omega$th Turing jump of $X$ for every set $X$. As a reversal we
obtain that $\RT(!\omega)$ is provable in Computable Mathematics ($\RCA_0$)
augmented by closure under the $\omega$ Turing jump. The system obtained from $\RCA_0$ by adding the axiom 
stating the closure under the $\omega$ Turing jump is denoted in the literature 
as $\ACA_0^+$. 

By analogy with $\RT(!\omega)$ we formulate and prove a version 
of Kanamori-McAloon's Regressive Ramsey Theorem \cite{Kan-McA:87} for regressive
colorings of exactly large sets and study its effective content. We prove analogous 
results as for $\RT(!\omega)$. 

In addition, we present a natural 
characterization of $\RT(!\omega)$ in terms of truth predicates over Peano Arithmetic. 

We believe that our results are interesting from the point of view of Computable Mathematics 
and of the Proof Theory of Arithmetic. By Computable Mathematics we here mean 
the task of measuring the computational complexity of solutions of computable 
instances of combinatorial problems. 
We give a complete characterization of the strength of $\RT(!\omega)$ in terms
of Computability Theory. Our results also yield a characterization of $\RT(!\omega)$
in terms of proof-theoretic strength as measured by equivalence to subsystems of second order
arithmetic, in the spirit of Reverse Mathematics. 
Ramsey's Theorem has been intensively studied from both the viewpoint of Computable
Mathematics and of the Proof Theory of Arithmetic, 
and our characterizations nicely extend the known relations between Ramsey Theorem 
for coloring finite hypergraphs and the finite Turing jump.
On the other hand, natural combinatorial theorems at the level of first-order arithmetical
truth are not common. Our results show that going from colorings of sets of a fixed 
finite cardinality to colorings of large sets correspondingly boosts 
the complexity of a coloring principle from hardness with respect to fixed levels 
of the arithmetical hierarchy to hardness with respect to the whole hierarchy. 
Thus, moving from finite dimensions to exactly large sets acts as a uniform 
transfer principle corresponding to the move from the finite Turing jumps to the 
$\omega$ Turing jump. It might be the case that a similar effect can be obtained 
in other computationally more tame contexts. 
We note that some natural isomorphism problems for computationally tame structures 
(e.g., the isomorphism problem for automatic graphs and for automatic 
linear orders) have been recently characterized 
as being at least as hard as $0^{(\omega)}$ (see \cite{Kus-Liu:LICS10}). 
Our results might have interesting connections with this line of research to the extent
that graph isomorphism can be related to homogeneity. 

\section{$\RT(!\omega)$ and Ramsey Theorem}
We first give a combinatorial proof of $\RT(!\omega)$ featuring an 
infinite iteration of the finite Ramsey Theorem. This proof will be used
as a model for our upper bound proof in Section \ref{sec:main}. 
We then recall what is known about the effective content of Ramsey Theorem 
and establish the easy fact that $\RT(!\omega)$ implies Ramsey Theorem for all finite exponents. 
We denote by $[X]^{!\omega}$ the set of exactly large subsets of $X$. For the rest
we follow standard partition-calculus notation from combinatorics.


\begin{proof}[Proof of Theorem \ref{thm:rtomega}]
Let $M$  be an infinite subset of $\Nat$, let $C:[\Nat\setminus\{1,\dots,a\}]^{!\omega}\to 2$. 
We build an infinite homogeneous subset $L\subseteq M$ for $C$ in stages. 
We keep in mind the fact that the family of all exactly large subsets of $M$ 
can be decomposed based on the minimum element of the set, in the sense 
that $S\in [\Nat]^{!\omega}$ if and only if $S=\{s_1,s_2,\dots,s_m\}$
and $\{s_2,\dots,s_m\}\in [\Nat - \{1,\dots,s_1\}]^{s_1}$. 

Let  $\isfunc{C_a}{[\Nat]^{a}}{2}$
be defined as $C_a(x_1,\dots,x_a)=C(a,x_1,\dots,x_a)$.
We define a sequence $\set{(a_i,X_i)}_{i\in\Nat}$ such that
\begin{itemize}
\item
$a_0= \min(M)$,
\item
$X_{i+1}\subseteq X_i\subseteq M$,
\item
$X_i$ is an infinite and $C_{a_i}$--homogeneous   and $a_i<\min (X_i)$,
\item
$a_{i+1}=\min X_i$.
\end{itemize}
At the $i$-th step of the construction we use Ramsey Theorem for coloring 
$a_i$--tuples from the infinite set $X_{i-1}$ (where $X_{-1}=M$).
We finally apply Ramsey Theorem for coloring singletons in two colors (i.e., the 
Infinite Pigeonhole Principle) to the sequence $\set{a_i}_{i\in\Nat}$ to 
get an infinite $C$--homogeneous set.  
\end{proof}

Note that the above proof ostensibly uses induction on $\Sigma^1_1$-formulas. 
We will show below how to transform the above proof into a proof using only induction on arithmetical 
formulas with second order parameters.

\medskip
We now recall what is known about the computational content of Ramsey Theorem and establish a first, easy
comparison with $\RT(!\omega)$. For $n\in\Nat$, we denote by $\RT^n$ the standard Ramsey Theorem for 
colorings of $n$-tuples in two colors, i.e., the assertion that every coloring $C$ of 
$[\Nat]^n$ in two colors admits an infinite homogeneous set. 
With a notable exception, the status of Ramsey's Theorem with respect to 
computational content is well-known, as summarized in the following theorems. 

\begin{theorem}[Jockusch, \cite{Joc:72}]\label{thm:jockusch}$\,$
\begin{enumerate}
\item For each $n\geq 2$ there exists a computable coloring $C:[\Nat]^n\to 2$ admitting 
no infinite homogeneous set in $\Sigma_n^0$.  
\item For each $n$, for each computable coloring $C:[\Nat]^n\to 2$, there exists
an infinite $C$-homogeneous set in $\Pi_n^0$.
\item For each $n\geq 2$ there exists a computable coloring $C:[\Nat]^n\to 2$ all of whose
homogeneous sets compute $0^{(n-2)}$. 
\end{enumerate}
\end{theorem}
Points (1), (2), (3) of the above Theorem are Theorem 5.1, Theorem 5.5 and Theorem 5.7 in \cite{Joc:72}, respectively.
Essentially drawing on the above results, Simpson proved the following Theorem (Theorem III.7.6 in \cite{Sim}). 
\begin{theorem}[Simpson, \cite{Sim}]
The following are equivalent over $\RCA_0$.
\begin{enumerate}
\item $\RT^3$,
\item $\RT^n$ for any $n\in\Nat$, $n\geq 3$,
\item $\forall X\exists Y (Y = X')$.
\end{enumerate}
\end{theorem}
In (3) above, the expression $\forall X\exists Y (Y = X')$
is a formalization of the assertion that the Turing jump of $X$ exists (and is $Y$). Details 
on how this formalization is carried out in $\RCA_0$ will be presented when needed. 
It is also known that the three statements of the previous Theorem are equivalent to 
the system $\ACA_0$ (i.e., the system obtained by adding to $\RCA_0$ all instances of the comprehension axiom 
for arithmetical formulas). One of the major open problems in the Proof Theory of Arithmetic is whether Ramsey's Theorem for 
colorings of pairs implies the totality of the Ackermann function over $\RCA_0$ (see \cite{See-Sla:95,
Cho-Joc-Sla:01}). 

The strength of the full Ramsey Theorem (with syntactic universal quantification over all exponents)
has been established by McAloon \cite{McA:85}. 
\begin{theorem}[McAloon, \cite{McA:85}]\label{thm:McAloon}
The following are equivalent over $\RCA_0$.
\begin{enumerate}
\item $\forall n \RT^n$,
\item $\forall n \forall X \exists Y(Y=X^{(n)})$.
\end{enumerate}
\end{theorem}
In (2) above the expression $\forall n \forall X \exists Y(Y=X^{(n)})$ denotes a formalization of the 
assertion that the $n$-th Turing jump of $X$ exists for all $n$. Details on 
how this formalization is carried out in $\RCA_0$ will be presented when needed.

Our main result --- Theorem \ref{thm:main} below --- is that an analogous relation holds between 
$\RT(!\omega)$ and closure under the $\omega$-jump. 
Theorem \ref{thm:McAloon} establishes the equivalence of $\forall n \RT^n$
with the system $\ACA_0'$ consisting of $\RCA_0$ augmented by an axiom stating that for every 
$n$ and for every set $X$ the $n$-th jump of $X$ exists for all sets $X$. 
As a corollary of our computability-theoretic analysis we will obtain 
that $\RT(!\omega)$ is equivalent to the system $\ACA_0^+$ consisting of $\RCA_0$ augmented by 
an axiom stating that for every set $X$ the $\omega$-jump of $X$ exists. 

The following easy Proposition relates $\RT(!\omega)$ to the standard Ramsey Theorem. 

\begin{proposition}\label{prop:impliesRamsey}
$\RT(!\omega)$ implies $\forall n \RT^n$ over $\RCA_0$.
\end{proposition}

\begin{proof}
Let $n\geq 1$ and $C:[\Nat]^n\to 2$ be given. We construct $C':[\Nat]^{!\omega}\to 2$ from $C$ as follows. 
Let $s=\{s_0,\dots,s_m\}$ be an exactly large set (then $m=s_0$). We
set 
$$
C'(s)=\begin{cases}
C(s_0,\dots,s_{n-1}) & \mbox{ if } s_0\geq n,\\
0 & \mbox{ otherwise.}
\end{cases}
$$
Let $H$ be an infinite $C'$-homogeneous set as given by $\RT(!\omega)$. 
Let $i\in \{0,1\}$ be the color of $[H]^{!\omega}$. 
Let $H' = H \cap [n,\infty)$. Let $s\in [H']^n$. Thus $\min(s)\geq n$. 
Let $s'$ be any exactly large set extending $s$ in $H'$. 
Then $C(s)=C'(s')=i$. Thus $H'$ is $C$-homogeneous of color $i$. 
\end{proof}
We will see below that $\RT(!\omega)$ is in fact strictly stronger than $\forall n \RT(n)$. 

\section{$\RT(!\omega)$ and Second Order Arithmetic with $\omega$-jumps}\label{sec:main}

We prove the following Theorem, characterizing the strength of $\RT(!\omega)$ over Computable 
Mathematics. 

\begin{theorem}\label{thm:main}
The following are equivalent over $\RCA_0$.
\begin{enumerate}
\item
$\RT(!\omega)$,
\item
$\forall X \exists Y(Y=X^{(\omega)})$.
\end{enumerate}
\end{theorem}
In (2) above, the expression $\forall X\exists Y (Y = X^{(\omega)})$
is a formalization of the assertion that the $\omega$th Turing jump of $X$ exists. Details 
on how this formalization is carried out in $\RCA_0$ will be presented when needed. 

The implication from $1.$ to $2.$ follows from Theorem \ref{thm:good} below. 
The implication from $2. $ to $1. $ follows from Theorem \ref{thm:upperbound} below. 
The system consisting of $\RCA_0$ plus the axiom $\forall X \exists Y (Y = X^{(\omega)})$ is known as $\ACA_0^+$.
From the viewpoint of Computable Mathematics, the implication from $1. $ to $2. $ 
is essentially based on a purely computability-theoretic result showing that $\RT(!\omega)$
has computable instances all of whose solutions compute $0^{(\omega)}$ (see Theorem
\ref{thm:good} and Proposition \ref{prop:hardcoloring} below). 

\subsection{Lower Bounds}
Our first result is that $\RT(!\omega)$ admits a computable instance 
that does not admit arithmetical solutions. This is obtained by a Shoenfield's Limit
Lemma construction based on the colorings from Jockusch's original proof of Theorem \ref{thm:jockusch} point
(1). Our second main result is that $\RT(!\omega)$ admits a computable instance all of whose solutions 
compute $0^{(\omega)}$. Recall that there exists 
sets that are incomparable with all $0^{(i)}$ with $i\geq 1$ (see, e.g., \cite{Rog}).

We actually prove that $\RT(!\omega)$ implies $\forall X\exists Y(Y=X^{(\omega)})$
over $\RCA_0$. Note that for the hardness result we do {\it not\/} use Jockusch's proof of Theorem \ref{thm:jockusch}
point (3) (i.e., essentially, Lemma 5.9 in \cite{Joc:72}). Instead we provide an explicit construction 
of a family of suitable colorings. The construction mimics some model-theoretic constructions of
indicators for classes of $\Sigma^0_n$ formulas. For a very nice and short introduction
into this method we refer to \cite{Kot:08}.  In addition, we show how to adapt 
the proof of Proposition 4.4 in the recent \cite{Dza-Hir:11} to get a computable instance of $\RT(!\omega)$
all of whose solutions compute all levels of the arithmetical hierarchy. 

We fix the following computability-theoretic notation. Let $\varphi$ be a fixed acceptable numbering \cite{Rog} for
a class of all recursive functions
\footnote{By definition, the {\em acceptable\/} programming systems for a class
are those which contain a universal simulator and into which all other universal programming
systems for the class can be compiled. Acceptable systems are characterized
as universal systems with an algorithmic substitutivity principle called S-m-n
and satisfy self-reference principles such as Recursion Theorems  
\cite{Rog}}.
We write $\{e\}^X(x)=y$ to indicate that the $\varphi$-program with index $e$ and oracle $X$ outputs $y$ on input $x$. 
We write $\{e\}^X(x)\stops$ if there exists a $y$ such that $\{e\}^X(x)=y$. 
Following notation from \cite{Soa} (Definition III 1.7), we write $\{e\}^X_s(x) =y$ if $x,y,e<s$ and $s>0$ and 
a $\varphi$-program with an index $e$ and oracle $X$ outputs $y$ on input $x$ within less than $s$ steps of 
computation and the computation only uses numbers smaller than $s$. We say that such an 
$s$ bounds {\it the use} of the computation. We occasionally 
write $\varphi_{e,s}^X(x) =y$ for $\{e\}^X_s(x) =y$. For the sake of our proof-theoretic results to follow we assume 
to have fixed a formalization of the assertion $\{e\}^X_s(x) =y$. 
We write $\{e\}^X_s(x)\stops$ (or $\varphi_{e,s}(x)\stops$)
if $\exists y(\{e\}^X_s(x) =y)$. $W_{e,s}^X$ denotes the domain of $\{e\}^X_s$. 
A set $X$ is Turing-reducible to a set $Y$ (denoted $X\leq_T Y$) if and only if there exist $i,j$ such that 
$(\forall x)(x\in Y \leftrightarrow \exists s (\{i\}^X_s(x)\stops))$
and $(\forall x)(x\notin Y \leftrightarrow \exists s (\{j\}^X_s(x)\stops))$.
Once a suitable formalization of the assertion $\{e\}^X_s(x) =y$ is fixed, the above
definition of $X\leq_T Y$ can be formalized in Computable Mathematics ($\RCA_0$).
We choose not to distinguish notationally between the real concept and its formalization, and
we define the two at once. We take care of defining the relevant computability-theoretic notions (e.g., the Turing jumps)
in such a way as to make it clear how they formalize in subsystems of second 
order arithmetic.


\bigskip
We first show how to define a computable coloring of exactly large sets such that all 
all homogeneous sets avoid all levels of the Arithmetical Hierarchy. Our first step towards this goal 
is the following relativized version of a result of Jockusch's \cite{Joc:72}.

\begin{lemma}\label{lem:Joc0}
There exists a $X$-computable coloring $\isfunc{e^X}{[\Nat]^2}{\set{0,1}}$ such that 
whenever $X$ is a $\Sigma_i^0$--complete set then $e^X$ has no homogeneous 
set in $\Sigma_{i+2}^0$.
\end{lemma}

\begin{proof}
A straightforward relativization of Theorem 3.1. of \cite{Joc:72}.
\end{proof}

In our construction below we make use of Shoenfield's Limit Lemma \cite{Sho:59}. This
result is usually stated as follows (see, e.g., \cite{Soa} for a standard textbook treatment). 
If $B$ is computably enumerable in $A$ and $f\leq_T B$ 
then there exists a binary $A$-computable function $h(x,s)$ such that $f(x)=\lim_{s} h(x,s)$, 
for every $x$. In our application below we will have $B=A'$. 
On the other hand, we will need more uniformity, as we now indicate. 
Let $g^X(i,e,s,x)$ be defined as follows. 
$$
g^X(i,e,s,x)=
\begin{cases}
\{e\}_s^{W_{i,s}^X}(x) & \mbox{ if } \{e\}_s^{W_{i,s}^X}(x)\stops,\\
0 & \mbox{ otherwise.}
\end{cases}
$$
For each fixed $X$, $g^X$ is $X$-computable. Let $B$ be computably enumerable in $A$
and let $f$ be computable in $B$. Let $i$ and $e$ be such that 
$B = W_i^A$ and $f=\{e\}^B$. Then 
$$ f(x) = \lim_s g^A(e,i,x,s).$$
In fact, in our application, we will have $B=K_{i+1}$ and $A=K_i$, where
$\set{K_i}_{i\in\Nat}$ is a fixed sequence of sets such that 
$K_0=\emptyset$ and, for each $i\geq 1$, $K_i$ is a $\Sigma^0_i$--complete set. 
For the sake of uniformity of our construction below, we take $K_{i+1}$ to be a halting problem for machines
with oracle $K_{i}$, for $i\geq 0$. So, e.g., $K_1$ is just the halting problem for standard Turing machines.
We fix an index $h$ such that for every $i\geq 0$, $K_{i+1}= W_h^{K_i}$.
In our application of Shoenfield's Limit Lemma to $B=K_{i+1}$ and $A=K_i$, we can thus
get rid of the argument $i$ in $g^X$ by freezing it to $h$ throughout.

\begin{theorem}\label{thm:sequence}
There exists a computable sequence of functions 
$e_n^X\colon [\Nat]^{n+2} \rightarrow \set{0,1}$
such that for any $n\geq 0$, for every $i\in\Nat$, $e^{K_i}_n$ is $K_i$-computable and 
computes a coloring with no homogeneous set in $\Sigma_{i+n+2}^0$.
\end{theorem}
\begin{proof}
We present a recursive procedure for constructing 
the sequence. For $n=0$ we take the function from Lemma \ref{lem:Joc0}.
Let us assume that we have defined a sequence with the desired properties up through $e^X_n$. 
We show how to compute the machine $\isfunc{e^X_{n+1}}{[N]^{n+3}}{\set{0,1}}$.

To ensure the desired properties of $e_{n+1}^X$ it is enough that 
for each $i\geq 0$ if $X=K_{i}$, any homogeneous set for $e_{n+1}^{K_{i}}$ 
is a homogeneous set for $e_{n}^{K_{i+1}}$. 
Moreover, $e^X_{n+1}$ should be obtained effectively from an index for $e_{n}^X$.

We use the same idea as in Proposition 2.1 of Jockusch' paper \cite{Joc:72}.
We take $g^X(e, x_1,\dots,x_{n+2},s)$ such that 
$$
\lim_{s\goesto\infty}g^{K_i}(e_n,x_1,\dots,x_{n+2},s)=e^{K_{i+1}}_n(x_1,\dots,x_{n+2}).
$$
As observed above, such $g^X$ is a fixed function. 
Now, we define $e^X_{n+1}$ as follows.
$$
e^{X}_{n+1}(x_1,\dots,x_{n+2},s) := g^{X}(e_n,x_1,\dots,x_{n+2},s).
$$
Now, if  $Y$ is an infinite homogeneous set for $e^{K_i}_{n+1}$ colored $0$, then it is easy to see that
any tuple $(x_1,\dots,x_{n+2})\in [Y]^{n+2}$ has to be colored $0$ by $e^{K_{i+1}}_n$ 
(and similarly for $Y$ colored $1$). This concludes the proof. 
\end{proof}

\begin{theorem}\label{thm:noSigmacoloring}
There exists a computable coloring $\isfunc{C}{[\Nat]^{!\omega}}{2}$
such that any infinite homogeneous set for $C$ is not $\Sigma^0_i$, for any $i\in \Nat$. 
\end{theorem}
\begin{proof}
Let $S=\set{s_1,\dots, s_{\card(S)}}$ be an exactly large set. Then $\card(S)=s_1+1$. 
We define 
$$
C(S)=e^{K_0}_{s_1-1}(s_1,\dots,s_{\card(S)}).
$$
Then any infinite homogeneous set $Y$ for $C$ has to be also homogeneous
for $e^{K_0}_{a-1}$, for each $a\in Y$. By Theorem \ref{thm:sequence} 
such a set is not in $\Sigma^0_{a+1}$.
Since $Y$ is infinite, $Y$ is not $\Sigma^0_i$, for any $i\geq 0$.
\end{proof}

We next show that for each set $A$ the principle $\RT(!\omega)$ has computable in $A$ instances all of whose solutions 
compute $A^{(\omega)}$. It follows as a corollary that $\RT(!\omega)$
proves over $\RCA_0$ that for every set $X$ the $\omega$-jump of $X$ exists. 

We give two proofs of this result. The construction in the first one 
mimics some indicator constructions for $\Sigma^0_n$ classes of formulas.
The second proof is obtained by adapting a recent proof by Dzhafarov and Hirst
\cite{Dza-Hir:11} in combination with an old result by Enderton and Putnam 
\cite{End-Put:70}. 

\begin{theorem}\label{thm:good}
For each set $A$ there exists a computable in $A$ coloring $C_\omega:[\Nat]^{!\omega}\to 2$ such that
all infinite homogeneous sets for $C_\omega$ compute $A^{(\omega)}$.
\end{theorem}

\begin{proof}
We fix the following definitions of Turing jumps for the sake of the present proof. 
For a set $X$ we denote by $X'$ the set 
of indices of Turing machines which stop on input $0$ with $X$ as an oracle:
$$
X'=\set{e: \set{e}^X(0){\stops}}.
$$
We denote the $n$-th jump of $X$ by $X^{(n)}$. For formalization issues, 
saying that `$X^{(n)}$ exists' is conveniently read as saying that 
there exists a set $X \subseteq \{0,\dots,n\}\times \Nat$ such that for each 
$i<n$, $\{a\,:\, (i+1,a)\in X\}$ is the jump of $\{b\,:\, (i,b)\in X\}$.

The $\omega$ jump, $X^{(\omega)}$, of a set $X$ is the set
$$
X^\omega=\set{ (i,j) : j \in X^{(i)}}.
$$
For formalization issues, saying that `$X^{(\omega)}$ exists' is conveniently 
read as saying that a set $Y$ exists such that, for all $n\in\Nat$, the $n$-th
projection of $Y$ is equal to $X^{(n)}$. 

Let $A$ be an arbitrary set.
We define a family of computable in $A$ colorings $C_n:[\Nat]^{n+1} \rightarrow \{0,1\}$, 
for $n\in \Nat$ and $n\geq 2$, and Turing machines $M_n(x,y)$ such that 
for any $n\geq 2$, the following three points hold. 
\begin{enumerate}
\item
All infinite homogeneous sets for $C_n$ have color $1$.
\item 
If $X$ is an infinite homogeneous set for $C_n$ then for any
$a_1 <\dots < a_{n+1}\in X$  it holds that if $a$ is a code for a sequence 
$(a_1,\dots, a_{n+1})$ then 
$M_n(x, a)$ decides $A^{(n-1)}$ for machines with indices less than or equal to $a_1$. 
\item
Machines $M_n$ are total. If their inputs are not from an infinite homogeneous set for $C_n$ then 
we have no guarantee on the correctness of their output.
\end{enumerate}
The second condition is a kind of uniformity condition. It states that no matter how we choose a sequence
$a=(a_1,\dots, a_{n-1})$ from an infinite $C_n$--homogeneous set we can decide $A^{(n-1)}$ below $a_1$
with one, recursively constructed machine $M_n$ which  is given $a$ as an oracle. 

We fix a pairing function $ \frac{x(x+1)}{2}+y$ which is a bijection between $\Nat^2$ and $\Nat$ and
denote it by $\pair{x,y}$. 
We define $C_2$ as
$$
C_2(k, y,z) = \left\{
\begin{array}{ll}
1 & \text{if }\forall e \leq k (\set{e}_y^{A}(0)\stops \pimp \set{e}_z^{A}(0)\stops)\\
0 & \text{otherwise.}
\end{array}
\right.
$$

Now, if $X$ is an infinite $C_2$--homogeneous set then it has to be colored $1$. 
If $k \in X$ then 
let us take a bound $b\in X$ such that for each Turing machine $e\leq  k$ 
$$
\set{e}^{A}(0)\stops \pimp \set{e}^{A}_b(0)\stops.
$$
Such a bound exists since $X$ is infinite and there are only finitely many Turing machines below $k$.
It follows that any $y\in X$ greater than $b$  
has the above property too. Therefore,   the color of any tuple $\{k,y, y'\}\in [X]^3$,
where $y,y'\geq b$  has to be $1$. It follows that the whole $X$ has to be colored $1$.

Let us also observe that it is easy to construct a machine $M_2(e,(k,b,b') )$ that 
searches for a computation of $e$ below $b$, provided that $e\leq k$.
Such a machine decides $A'$ up to $k$ if it is given $k$ and $b >k $
which belongs to some infinite $C_2$--homogeneous set.

Now, let us assume that we have constructed $C_n$ and $M_n$ for some $n\geq 2$. We obtain
$C_{n+1}$ and $M_{n+1}$ as follows. We set $C_{n+1}(a_1,\dots, a_{n+2})=$
$$\begin{cases}
1 & \mbox{if  }\{a_1,\dots, a_{n+2}\} \mbox{ is } C_n\mbox{--homogeneous and}\\
    &   \forall e\leq a_1 (\set{e}^{Y}_{a_2}(0)\stops\pimp \set{e}^{Y}_{a_3}(0)\stops),\text{ where}\\
      & Y=\set{i \leq a_2 : M_n(i,(a_2,\dots, a_{n+2}))  \text{ accepts,}}\\
0 & \text{otherwise}.
\end{cases}
$$

Ideally, we would like to replace the condition in the second line of  the above definition
 by
$$
\forall e\leq a_1 (\set{e}^{A^{(n-1)}}_{a_2}(0)\stops\pimp \set{e}^{A^{(n-1)}}_{a_3}(0)\stops).
$$
However, such a condition would lead to a  coloring which may be non-recursive in $A$. 
Thus, instead of checking $\set{e}^{A^{(n-1)}}_{z}(0)\stops$
we use approximations of these sets computed by machines~$M_{n}$.

Now, let an infinite set $X$ be $C_{n+1}$--homogeneous and assume, towards a contradiction, 
that it is colored $0$. Let us take an infinite $Z\subseteq X$ such that $Z$ is colored $1$ by $C_n$. 
For a given $a_1\in Z$ let $a_2$ be so large that $M_n$ can correctly decide all oracles queries for machines below $a_1$ 
on input $0$. Let us take $a_3,\dots, a_{n+2}\in Z$ such that
$$
\forall  e \leq a_1 (\set{e}^Y_{a_2}(0)\stops \pimp \set{e}^Y_{a_3}(0)\stops),
$$
where $Y=\set{i \leq a_2 : M_n(i,(a_2,\dots, a_{n+2}))  \text{ accepts}}$.
Again, such $a_2,\dots, a_{n+1}$ exists since there are only finitely many machines below $a_1$ and 
$M_n(i, (a_1,\dots,a_{n+2}))$ correctly decides $A^{(n-1)}$ below $a_2$. Thus, we have equivalence
$$
\forall  e \leq a_1 ( \set{e}^Y_{a_2}(0)\stops  \pimp \set{e}^{A^{(n-1)}}(0)\stops).
$$
Now, it is easy to see that the color of $C_{n+1}(a_1,\dots, a_{n+2})=1$ and, consequently,
the whole $X$ is colored $1$.

Now, let us describe a Turing machine $M_{n+1}(e, (a_1,\dots, a_{n+2}))$ which decides $A^{(n)}$  below
$a_1$ if $(a_1,\dots, a_{n+2})$ is a sequence from an infinite $C_{n+1}$--homogeneous set.
We use the fact that for each $a_1 <a_2$ from an infinite $C_{n+1}$--homogeneous set 
and for all $e<a_1$ we have
$$
\set{e}^{A^{(n-1)}}_{a_1}(0)\stops  \pimp \set{e}^{A^{(n-1)}}_{a_2}(0)\stops
$$
and consequently, by infinity of the given $C_{n+1}$--homogeneous set,
$$
\set{e}^{A^{(n-1)}}_{a_1}(0)\stops  \pimp \set{e}^{A^{(n-1)}}(0)\stops.
$$

In the first part of the computation $M_{n+1}(e, (a_1,\dots, a_{n+2}))$ computes the set 
$$
Y=\set{ i \leq a_2: \text{$M_n(i, (a_2,\dots, a_{n+1}))$ accepts}}.
$$
Then, it checks whether $\set{e}^Y_{a_2}\stops$ and if this holds, $M_{n+1}$ accepts.

Now, we may turn our attention to colorings of $!\omega$--large sets.
We construct a computable coloring $C_\omega$ and a Turing machine $M_\omega(e,a)$ such that
\begin{enumerate}
\item
All infinite homogeneous sets for $C_\omega$ are colored $1$.
\item 
If $X$ is an infinite homogeneous set for $C_\omega$ then for any
for any $a_1 <\dots < a_{k}\in X$  it holds that if $\set{a_1,\dots, a_k}$ is an exactly $\omega$--large set and 
$a$ is a code for the sequence 
$(a_1,\dots, a_{k})$ then 
$M_\omega(x, a)$ decides $A^{(\omega)}$ for pairs $(i,j)$ such that  $i, j\leq a_1$. 
\item
Machine $M_\omega$ stops on all inputs. If the inputs are not from an infinite homogeneous set for $C_\omega$ then
we have no guarantee on the correctness of the output.
\end{enumerate}

We define  $C_\omega$ as follows.
$$
C_{\omega}(a_1,\dots, a_k) = C_{a_1}(a_1,\dots, a_k).
$$
For a sequence $a=(a_1,\dots, a_k)$, we define
$$
M_{\omega}(e, a) = M_{a_1}(e, a).
$$
Since any infinite $C_\omega$--homogeneous set $X$ is also $C_n$--homogeneous 
for any $n\in \Nat$ one can easily show that $C_\omega$ and $M_\omega$ have the required properties.

Finally, we define a machine $M(x)$ which decides $A^{(\omega)}$ with any infinite $C_\omega$--homogeneous set $X$ 
given as an oracle. Let us fix a recursive sequence of recursive functions $f_{i,j}$, for $i\leq j$, such that $f_{i,j}$ is a many--one reduction from $A^{(i)}$ to $A^{(j)}$. The machine $M$ on input $(i,j)$ searches for an element $a_1 \in X$ such that $i,j <a_1$ .
Then, it searches for the next $a_1$ elements of $X$, be they $a_2, \dots, a_k$. After constructing such a sequence
$M$ simulates $M_{a_1}( f_{i,a_1}(j), (a_1,\dots, a_k))$ and outputs the result of this simulation.

Let us observe that if we want $M$ to be provably total in some theory $T$ we need $T$ to prove
that for each infinite set $X$  and for each $y$ there exists an $\omega$--large subset of $X$ 
with $y$ as a minimum. But this is obviously true in the case of $\omega$--large sets.
\end{proof}

It is interesting to observe that a proof from the recent \cite{Dza-Hir:11} can be easily adapted 
to show that $\RT(!\omega)$ has a computable instance all of whose solutions compute $0^{(i)}$ for all $i\in\Nat$. 
This gives, in combination with a property of least upper bounds of sequences of degrees as we will see, 
an alternative proof of our Reverse Mathematics corollary of Theorem \ref{thm:good}. We now give the necessary 
details, which illustrate a strict analogy between model-theoretic-like constructions as in the proof of 
Theorem \ref{thm:good} and computability-theoretic constructions. 

The proof of the following proposition is modeled after the proof of Proposition 4.4 
in \cite{Dza-Hir:11}. Although the latter proof is for a different principle
(the so-called Polarized Ramsey Theorem), the gist of it is to show directly that 
$\forall n\RT^n$ implies $\forall n \forall X \exists Y (Y=X^{(n)})$ without the need of formalizing the proof of 
Theorem \ref{thm:jockusch} point (3). This turns out to be surprisingly well-suited for our purposes. 
We denote by $2\Nat$ the set of even natural numbers. Note that the proof can be carried out in 
$\ACA_0'$, which is available under the assumption of $\RT(!\omega)$ by virtue of Proposition \ref{prop:impliesRamsey}. 

\begin{proposition}\label{prop:hardcoloring}
For every set $X$ there exists a computable coloring $C^X:[\Nat]^{!\omega}\to 2$ such that
if $H\subseteq 2\Nat$ is an infinite homogeneous set for $C$ then
$H$ computes $X^{(n-1)}$ for every $2n\in H$. 
\end{proposition}

\begin{proof}
For the sake of the present argument we define/formalize the assertion $Y=X'$ stating that 
$Y$ is the Turing jump of $X$ as follows. 
$$\forall x \forall e (\langle x,e\rangle \in Y \leftrightarrow \exists s (\{e\}^X_s(x)\stops)$$ 
The definition of the $n$th jump is then as in the proof of Theorem \ref{thm:good}. 
Following \cite{Dza-Hir:11} we define
the following approximations of the finite jumps (where \cite{Dza-Hir:11} use $\Phi$ 
we use $W$, $\Phi$ being traditionally reserved for Blum Complexity Measures). 
For any set $X$ and integer $s$ define
$$ X'_s = \{ \langle m,e\rangle \;:\; (\exists t < s) m \in W_{e,t}^X\}.$$
For integers $u_1,\dots,u_n$ and $s$ define $X^{(0)}=X$, and 
$$ X^{(n+1)}_{u_n,\dots,u_1,s} = (X^{(n)}_{u_n,\dots,u_1})'_s.$$

$C^X$ is defined as follows. Let $A = \{a_0,\dots,a_p\}$ be exactly large, 
i.e., $a_0 = p$. If $a_0 = 2n$ for some $n$ let $C^X(A)=1$ if there exist
$1\leq i \leq n$ and $\exists (e,m) < a_{n-i}$ such that 
$$ \neg ((m,e) \in X^{(i)}_{a_n,\dots,a_{n-i+1}} \leftrightarrow (m,e) \in X^{(i)}_{a_{2n},\dots,a_{2n-i+1}})$$
and $C^X(A)=0$ otherwise. If $a_0 = 2n+1$ then $C^X(A)=0$ (the value is irrelevant in 
this case). Let $H$ be an infinite homogeneous set for $C^X$ as given by $\RT(!\omega)$
applied to $C^X$ and $M=[2,\infty)\cap 2\Nat$. 

\medskip
We first claim that the color of $C^X$ on $[H]^{!\omega}$ is $0$. Suppose otherwise. 
Let $A\in [H]^{!\omega}$ such that $C^X(A)= 1$. Then there exists $i\leq n$ such that
$\exists (e,m) < a_{n-i}$ such that 
$$ \neg ((m,e) \in X^i_{a_n,\dots,a_{n-i+1}} \pimp (m,e) \in X^i_{a_{2n},\dots,a_{2n-i+1}})$$
where $n$ is such that $A=\{2n,a_1,\dots,a_{2n}\}$. Now consider the coloring
obtained by coloring $B=\{b_1,\dots,b_{2n}\}\in [H\cap (2n,\infty)]^{2n}$ with the
least $i\leq n$ such that
$\exists (e,m) < b_{n-i}$ such that 
$$ \neg ((m,e) \in X^{(i)}_{b_n,\dots,b_{n-i+1}} \pimp (m,e) \in X^{(i)}_{b_{2n},\dots,b_{2n-i+1}}).$$
By Ramsey Theorem $\RT^{2n}_n$, this coloring admits
an infinite homogeneous set $H'\subseteq H\cap (2n,\infty)$. Then we argue exactly as
in \cite{Dza-Hir:11} to obtain a contradiction.

\medskip
Now we claim that for every $h\in H$, $X^{(n-1)}$ is computable in $H$, where $h = 2n$. 
In fact we show that $X^{(n-1)}$ is definable by recursive comprehension from $H$. We define a finite sequence
$(X_0,\dots,X_{n-1})$ as follows. $X_0=X$. For each $i\in [1,n)$, $(m,e)\in X_i$ if and only 
if $(m,e) \in X^{(i)}_{a_n,a_{n-1},\dots,a_{n-i+1}}$ where $(2n,a_1,\dots,a_n,a_{n+1},\dots,a_{2n})$ is
the lexicographically least exactly large set in $H$ such that $(m,e) < a_{n-i}$. 

We claim that for each $i < n-1$, $X_{i+1}=X_i'$. 

First we show that $X_{i+1}\subseteq X_i'$. Suppose $(m,e)\in X_{i+1}$. By definition 
of $X_{i+1}$, $(m,e) \in X^{(i+1)}_{a_n,a_{n-1},\dots,a_{n-i}}$ where $(2n,a_1,\dots,a_n,a_{n+1},\dots,a_{2n})$ is
the lexicographically least exactly large set in $H$ such that $(m,e) < a_{n-i-1}$.
Thus $(m,e)\in (X^{(i)}_{a_n,\dots,a_{n-i+1}})'_{a_{n-i}}$, 
and so $(\exists t < a_{n-i})(m\in W_{e,t}^{X^{(i)}_{a_n,\dots,a_{n-i+1}}})$. Since $a_{n-i}$ bounds the use of the computation, 
and by homogeneity of $H$, it follows that $X^{(i)}_{a_n,\dots,a_{n-i+1}}$ and $X_i$ agree below $a_{n-i}$. 
Therefore $(\exists t < a_{n-i})(m\in W_{e,t}^{X_i})$.  

Next we show that $X_i'\subseteq X_{i+1}$. Suppose $(m,e)\in X_i'$. Then there exists $t$ such that
$m\in W_{e,t}^{X_i}$. Let $(2n,a_1,\dots,a_n,a_{n+1},\dots,a_{2n})$ be
the lexicographically least exactly large set in $H$ such that $(m,e) < a_{n-i-1}$.
Choose $b_{n-i} \in H$ such that $b_{n-i} > \max\{t,a_{n-i-1}\}$. Choose an increasing tuple
$(b_{n-i+1},\dots,b_n)$ in $H$ with $b_{n-i} < b_{n-i+1}$. By the homogeneity of $H$
and the definition of $X_i$, the sets $X_i$ and $X^{(i)}_{b_n,\dots,b_{n-i+1}}$ agree on elements below
$b_{n-i}$. Thus $(\exists w < b_{n-i})(m\in W_{e,t}^{b_n,\dots,b_{n-i+1}})$, i.e., 
$(m,e)\in (X^{(i)}_{b_n,\dots,b_{n-i+1}})'_{b_{n-i}}$, and the latter set is equal to 
$X^{(i+1)}_{b_n,\dots,b_{n-i}}$. By homogeneity of $H$ we then have that
$(m,e) \in X^{(i+1)}_{a_n,\dots,a_{n-i}}$, hence $(m,e)\in X_{i+1}$. 
\end{proof}

It is well-known that $\{ 0^{(i)} \,:\, i\in\Nat\}$ has no least upper bound. Yet we can obtain from the previous
proposition a result about $0^{(\omega)}$ by the following result by Enderton and Putnam
\cite{End-Put:70}. The original proof needs the existence of the double jump, which is well-within 
our current base theory by virtue of Proposition \ref{prop:impliesRamsey}. 

\begin{lemma}[Enderton-Putnam, \cite{End-Put:70}]\label{lem:EP70}
Let $I$ be an infinite set. Let $X$ be a set. Let $Y$ be a set such that for every 
$i\in I$, $X^{(i)} \leq_T Y$. Then, $X^{(\omega)}$ is many-one reducible to $Y^{(2)}$.
\end{lemma}

We can now derive our main proof-theoretical result of the present section. 
\begin{theorem}\label{thm:omegahard}
$\RT(!\omega)$ implies $\forall X \exists Y (Y=X^{(\omega)})$ over $\RCA_0$.
\end{theorem}

\begin{proof}
The result can be obtained by formalization of the proof of Theorem \ref{thm:good}. 

Alternatively, we can argue as follows. The proof of Proposition \ref{prop:hardcoloring} is so devised as to formalize in 
$\ACA_0'$ which is well within our hypotheses $\RCA_0+\RT(!\omega)$ (see Proposition \ref{prop:impliesRamsey}. 
Let $X$ be a computable set and $C^X$ be as in Proposition \ref{prop:hardcoloring}.
Then, by that proposition, every homogeneous set for $H$ computes $0^{(i)}$ for all $i\in\Nat$. 
Let $H$ be such an infinite homogeneous set for $C^X$. Such an $H$ exists by $\RT(!\omega)$ 
applied to the instance $(2\Nat,C^X)$. Then by Lemma \ref{lem:EP70}, $H^{(2)}$ computes $X^{(\omega)}$. 
So it remains to show that $\RT(!\omega)$ implies that $H^{(2)}$ exists. But this is obvious
since $\RT(!\omega)$ implies $\forall n \RT^n$, by Proposition \ref{prop:impliesRamsey}, 
and $\forall n \RT^n$ implies $\forall X \forall n \exists Y (Y = X^{(n)})$, by Theorem 
\ref{thm:McAloon}.
\end{proof}

\subsection{Upper Bounds}

We show a reversal of Theorem \ref{thm:omegahard}.
\begin{theorem}\label{thm:upperbound}
$\forall X \exists Y (Y = X^{(\omega)})$ implies $\RT(!\omega)$ over $\RCA_0$.
\end{theorem}

The idea of the proof is the following. We take the proof of 
$\RT(!\omega)$ in Theorem \ref{thm:rtomega} as a starting point. 
We replace the sets $X_i$ by Turing machines with oracles from $C^{(a)}$, 
for $a$ an element of a model of $\RCA_0$. These Turing machines are constructed in a uniform way.
These machines are designed so as to compute the sets $X_i$ and thus turn the induction 
in the proof of Theorem \ref{thm:rtomega} into a first-order induction.
Moreover, since they will need as oracles the sets $C^{(a)}$ the whole construction 
will be recursive in $C^{(\omega)}$.

The Lemma below presents the basic construction which replaces the use of sets $X_i$ by constructing
Turing machines with oracles. We do not tailor for optimality of the oracles used, rather for clarity of 
the construction and we only take care that all oracles used are below $C^{(\omega)}$.
We begin by recalling the definition of the Erd\H{o}s-Rado tree associated to a coloring.

\begin{definition}[Erd\H{o}s-Rado tree]
Let $a\geq 1$. Let $C:[\Nat]^{a+1}\to 2$. The Erd\H{o}s-Rado tree $T$ of $C$ is the set of 
finite sequences $t$ of natural numbers defined as follows. 
If $t$ is of length $\ell > n$, $t(n)$ is the least $j$ such that the following two conditions 
hold.
\begin{enumerate}
\item For all $m<n$, $t(m)<j$, and
\item For all $m_1< \dots < m_a < m \leq n$, $C(t(m_1),\dots, t(m_a),j) = C(t(m_1),\dots,t(m_a),t(m))$.
\end{enumerate}
\end{definition}
It is easy to see that $T$ is a finitely branching tree and computable in $C$. We denote by $A\oplus B$ the 
join of $A$ and $B$. 

\begin{lemma}\label{lem:upperbound}
Let $a\geq 1$. Let $\isfunc{C}{[U]^a}{2}$. One can find effectively a machine $f_a$ with oracle 
$(C\oplus U)^{(2a)}$ such that $f_a$ computes a $C$--homogeneous set.
\end{lemma}
\begin{proof}
For $a=1$, the machine $f_1$ needs to ask the $\Pi^0_2(C\oplus U)$ oracle whether 
$\forall n\exists k\geq n (C(k)=0\land U(k))$. If the answer is yes, then $f_1$ computes the set
$C(x)=0\land U(x)$, otherwise it computes the set $C(x)=1\land U(x)$.

Now, let us consider the induction step for $a+1$. Machine $f_{a+1}$ first 
constructs the Erd\H{o}s--Rado tree $T_a$ for the function $\isfunc{C}{[U]^{a+1}}{2}$.
The tree $T_a$ is computable in $C\oplus U$. Then, we can obtain an index for a machine $e_p$ which 
computes the leftmost infinite path $P$ of $T_a$ using a $\Pi^0_2(C\oplus U)$--complete oracle.
Indeed, a sequence $\langle b_0,\dots, b_k\rangle \in P$ if and only if
$$
\forall n\geq k \exists \langle b_{k+1},\dots,b_n\rangle\textrm{\ such that\ } \langle b_0,\dots,b_n\rangle \in T_a
$$
and 
$\exists n \geq k$ such that $\forall \langle b'_0,\dots,b'_k\rangle \leq_{\textrm{lex}}\langle b_0,\dots,b_k\rangle$, 
$\forall \langle b'_{k+1},\dots, b'_{n}\rangle$, the following holds
$$\langle b'_0,\dots, b'_n\rangle \not \in T_a.$$
The crucial property of elements from $P$ is that the color of any $(a+1)$--tuple from $P$ does not 
depend on the last element of the tuple. Thus, if we restrict the domain of the coloring $C$ to $P$, we can treat the coloring
$C$ as a coloring of $a$--tuples. Let us call this restricted coloring $C'$. 
Then, we construct a machine $f_a$ (which may be obtained by inductive hypothesis) and use it with oracle $(C'\oplus P)^{(2a)}$. 
Any infinite $C'$--homogeneous subset of $P$ computed by $f_a$ is also $C$--homogeneous.
Moreover, since $P$ is recursive in $\Pi^0_2(C\oplus U)$,
the complexity of the oracle is $(C\oplus U)^{(2(a+1))}$ as required. This completes the recursive 
construction and the proof of the Lemma.
\end{proof}

\begin{proof}[Proof of Theorem \ref{thm:upperbound}]
Once the machine $f_a$ are constructed as in Lemma \ref{lem:upperbound} we can replace oracles they use by
one oracle $C^{(\omega)}$. 
At each step of the construction machines query only a finite fragments of $C^{(\omega)}$  but to make a construction
uniform we can replace calls to different oracles by calls to $C^{(\omega)}$.

Now, we can replace the $\Sigma^1_1$--induction in the proof of Theorem \ref{thm:rtomega} by first-order induction. 
As in the proof of Theorem \ref{thm:rtomega}, for a coloring $C\colon[\bN]^{!\omega}\rightarrow 2$ we define $C_a(x_1,\dots, x_a)= C(a,x_1,\dots, x_a)$, for $a< x_1<\dots < x_a$. If $f_a$ is a function that computes $C_a$ homogeneous set,
we can refer to this set as the range of $f_a$, ${\rm rg}(f_a)$. We formulate the first order induction in the following form:
for each $n$ there exists a sequence $\set{ (a_i, f_{a_i}) : i\leq n}$ such that for each $i<n$,
\begin{itemize}
\item
$a_0=2$,
\item
${\rm rg}(f_{a_{i+1}}) \subseteq {\rm rg}(f_{a_i}) \subseteq \bN$,
\item
${\rm rg}(f_{a_i})$ is infinite and $C_{a_i}$--homogenous,
\item
$a_{i+1} = \min({\rm rg}(f_{a_i}) \cap \set{ x \in \bN \colon x>a_i})$.
\end{itemize}
The reader may want to compare these conditions with the conditions used in the proof of Theorem \ref{thm:rtomega} (cfr. second column of page 2).
Instead of sets $X_i$ we 
use indexes of machines $f_{a_i}$ computing  $C_{a_i}$--homogeneous sets.
Then, using arithmetical comprehension (which is available since $\forall X \exists Y (Y = X^{(\omega)})$ implies $\ACA_0$ and more) we may carry out the induction and prove  that there exists infinite sequence $\set{(a_i,f_{a_i})}_{i\in\Nat}$ with the above properties. By construction, the set 
the set $\{a_i: i\in \Nat\}$ is $C$--homogeneous.
\end{proof}

Let us observe than we could not carry out the above proof from the assumption $\forall n \RT^n$ 
even though we could perform {\it each step} of the induction. The problem is that we would
not have just one oracle $C^{(\omega)}$ in the whole construction but we could be forced 
to use stronger and stronger oracles at each step. So, the construction could not
be expressed as a single arithmetical formula.

\section{The Regressive Ramsey Theorem for coloring exactly large sets}

In this section we formulate and analyze an analogue of $\RT(!\omega)$ based
on Kanamori-McAloon's principle (also known as the Regressive Ramsey Theorem) \cite{Kan-McA:87}. 
This principle is well-studied (see, e.g., 
\cite{Mil:08,Koj-She:99,Koj-Lee-Omr-Wei:08,Car-Lee-Wei:11}) and is one of the most natural examples 
of a combinatorial statement independent of Peano Arithmetic. 
The idea for studying the analogue principle for colorings of exactly large sets
came from the analysis of the proof of Proposition 4.4 in \cite{Dza-Hir:11}. 
The natural way of glueing together the colorings used in that proof gives rise to a 
regressive function on exactly large sets. 

To state the Regressive Ramsey Theorem we need a bit of terminology. A coloring $C$ is 
called {\it regressive} if for every $S\subseteq \Nat$ of the appropriate type, $C(S)<\min(S)$
whenever $\min(S)>0$. We denote by $\KM^d$ the following statement: For every regressive 
coloring $C:[\Nat]^d\to \Nat$ there exists an infinite $H\subseteq \Nat$ such that 
the color of elements of $[H]^d$ only depends on their minimum, i.e., if $s,s'\in [H]^d$
are such that $\min(s)=\min(s')$ then $C(s)=C(s')$. A set such as $H$ is called {\it min-homogeneous}
for $C$. 

We formulate a natural infinite version of the Regressive Ramsey Theorem for coloring large 
sets as follows. For every regressive coloring $C:[\Nat]^{!\omega}\to \Nat$, for every 
infinite $M\subseteq \Nat$, there exists an infinite $H\subseteq M$ that is min-homogeneous
for $C$. We denote this statement by $\KM(!\omega)$. A combinatorial proof of $\KM(!\omega)$ 
can be given along exactly the same lines as the proof 
of $\RT(!\omega)$ in Theorem \ref{thm:rtomega} above.

In fact, as we now prove, $\KM(!\omega)$ is equivalent to $\RT(!\omega)$ over $\RCA_0$. 

\begin{proposition}
Over $\RCA_0$, $\KM(!\omega)$ and $\RT(!\omega)$ are equivalent.
\end{proposition}

\begin{proof}
We first prove that $\KM(!\omega)$ implies $\RT(!\omega)$. This is almost trivial. Let 
$C:[\Nat]^{!\omega}\to 2$ be given. Then $C$ is regressive on $[\Nat \setminus\{0,1\}]^{!\omega}$. Let $H$ be an infinite
min-homogeneous set for $C$. Define $C':[H]\to 2$ as follows. $C'(h)=i$ if 
all exactly large sets in $H$ with minimum $h$ have color $i$. By the Infinite Pigeonhole
Principle, let $H'\subseteq H$ be an infinite $C'$-homogeneous set. Then $H'$ is 
$C$-homogeneous. 

Now, we  prove that $\RT(!\omega)$ implies $\KM(!\omega)$. Let $C:[\Nat]^{!\omega}\to \Nat$
be a regressive coloring. We define $C':[\Nat]^{!\omega}\to \{0,1\}$ in  such  a way that 
if $X$ is an infinite $C'$--homogenous set then
$Y=\set{ x-1 \colon x\in X}$  is $\min$--homogenous for $C$.
For a tuple  $A=(a_0, \dots, a_k)\in [\bN]^{k+1}$, where $a_0\geq 1$ and $k=a_0$,
we define $C'(A)$ as $1$ if  all tuples  $(a_0 -1, c_1, \dots, c_{k-1}) \in [\set{a_i - 1 \colon 0\leq i \leq k}]^{a_0)}$
get the same color under $C$. Otherwise, we define $C'(A)$ as $0$. (We define $C'((0))$ arbitrarily.)
It is easy to prove by $\RCA_0$ induction that for each infinite $C'$--homogenous set $X$
has the stated above property.
\end{proof}

It is instructive to observe how the proof of Proposition \ref{prop:hardcoloring} goes through 
almost unchanged. The details diverge from the proof of Proposition 4.4. in \cite{Dza-Hir:11}
in a different point.

\begin{proposition}\label{prop:KMhardcoloring}
For every set $X$ there exists a computable regressive coloring $C^X:[\Nat]^{!\omega}\to 2$ such that
if $H\subseteq 2\Nat$ is an infinite min-homogeneous set for $C$ then
$H$ computes $X^{(n-1)}$ for every $2n\in H$. 
\end{proposition}

\begin{proof}
Let $X$ be a set. Define $C^X:[\Nat]^{!\omega}\to \Nat$ as follows. 

Let $A = \{a_0,\dots,a_p\}$ be exactly large, 
i.e., $a_0 = p$. 

If $a_0 = 2n$ for some $n$ then $A=\{2n,a_1,\dots,a_n,a_{n+1},\dots,a_{2n}\}$. 
Let $C^X(A)$ be the least $i\in [1,n]$ 
such that there exists $(m,e) < a_{n-i}$ such that 
$$ \neg ((m,e) \in X^i_{a_n,\dots,a_{n-i+1}} \pimp (m,e) \in X^i_{a_{2n},\dots,a_{2n-i+1}})$$
if such an $i$ exists, and $C^X(A)=0$ otherwise. If $a_0 = 2n+1$ then $C^X(A)=0$ 
(the value is irrelevant in this case). Note that $C^X$ is a regressive coloring. 
Let $H$ be an infinite min-homogeneous set for $C^X$ as given by $\KM(!\omega)$ applied to $C^X$ 
and $M=[2,\infty)\cap 2\Nat$. 

\medskip
\begin{sloppypar}
We first claim that the color of $C^X$ restricted to $[H]^{!\omega}$ is $0$ and 
$H$ is indeed homogeneous. 
Suppose otherwise by way of contradiction. Let $i>0$ be such that for some 
$n$, $A=\{2n,a_1,\dots,a_n,a_{n+1},\dots,a_{2n}\}\in [H]^{!\omega}$ and $C^X(A)=i$. 
Note that $i \leq n$ and that the color is $i$ for every exactly large set in $H$
with minimum $2n$. Let $H=\{2n_j\}_{j\in J}$ for some $J$. Let $n=n_{j}$. 
Let $h=n_{j+n-i}$. 
We claim that there exists $(m_0,e_0) < 2h$ such that for all $B\in [H\cap (2n,\infty)]^{2n}$
$$ \neg ((m_0,e_0)\in X^i_{b_n,\dots,b_{n-i+1}} \pimp (m_0,e_0)\in X^i_{b_{2n},\dots,b_{2n-i+1}}) $$
where $B=\{b_1,\dots,b_n,b_{n+1},\dots,b_{2n}\}$. We get the existence of $(m_0,e_0)$
by coloring $[H\cap (2n,\infty)]^{2n}$ according to the least $(m,e) < 2h$ witnessing the color is 
$i$ (i.e., by an application of a finite Ramsey Theorem of suitable dimension). 
\end{sloppypar}

Fix such a $B$. By minimality of $i$ it must be the case that $X^i_{b_n,\dots,b_{n-i+2}}$
agrees with $X^i_{b_{2n},\dots,b_{2n-i+2}}$ on values below $b_{n-i+1}$. Therefore 
$$(m_0,e_0)\in X^i_{b_n,\dots,b_{n-i+1}} \rightarrow (m_0,e_0)\in X^i_{b_{2n},\dots,b_{2n-i+1}},$$ 
since $b_{n-i+1} < b_{2n-i+1}$. Then by choice of $(m_0,e_0)$ the converse implication must fail. 
Therefore $(m_0,e_0)\in X^{i}_{b_{2n},\dots,b_{2n-i+1}}$ holds unconditionally. Thus, 
$$(\exists t < b_{2n-i+1})(m_0\in W_{e_0,t}^{X^{i-1}_{b_{2n},\dots,b_{2n-i+2}}}).$$
Choose $(b_1^*,\dots,b_n^*,b_{n+1}^*,\dots,b_{2n}^*)$ in $[H\cap (2n,\infty)]^{2n}$
with $b_{n-i+1}^* > b_{2n-i+1}$ and $b_{2n-i+2}^* \geq b_{2n-i+2}$. By the same argument
as above applied to the sequence $(b_1,\dots,b_{2n-i+1},b_{2n-i+2}^*,\dots,b_{2n}^*)$
we have that 
$$(\exists t < b_{2n-i+1})(m_0\in W_{e_0,t}^{X^{i-1}_{b_{2n}^*,\dots,b_{2n-i+2}^*}}).$$
But by minimality of $i$ we have that $X^{i-1}_{b_n^*,\dots,b_{n-i+2}^*}$ and $X^{i-1}_{b_{2n}^*,\dots,
b_{2n-i+2}^*}$ must agree on 
all elements below $b_{n-i+1}^*$ and therefore also on all elements below $b_{2n-i+1}$. 
But $b_{2n-i+1}$ bounds the use of the computation showing $m_0\in W_{e_0}^{X^{i-1}}$
and since $b_{2n-i+1} < b_{n-i+1}^*$ we have that
$$ (\exists t < b_{n+i-1}^*)(m_0\in W_{e_0,t}^{X^{i-1}_{b_{n}^*,\dots,b_{n-i+2}^*}}),$$
and on the other hand, since $b_{2n-i+2}^* \geq b_{2n-i+2}$, we have that
$$ (\exists t < b_{2n+i-1}^*)(m_0\in W_{e_0,t}^{X^{i-1}_{b_{2n}^*,\dots,b_{2n-i+2}^*}}).$$
But these two facts contradict the choice of $(m_0,e_0)$. 

\medskip
We then claim that for every $h\in H$, $X^n$ is computable in $H$, where $h = 2n$. 
Since $H$ is homogeneous of color $0$, the argument goes through unchanged as in the proof of Proposition 
\ref{prop:hardcoloring}. 
\end{proof}

We next observe without proof that an analogue of Proposition \ref{thm:upperbound} holds for $\KM(!\omega)$. 
The proof is similar to that of Theorem \ref{thm:upperbound}.

\begin{theorem}\label{thm:KMupperbound}
$\forall X \exists Y (Y = X^{(\omega)})$ implies $\KM(!\omega)$ over $\RCA_0$.
\end{theorem}

\section{Peano Arithmetic with $\omega$ inductive truth predicates}

In this section we compare the strength of $\RT(!\omega)$ with Peano Arithmetic 
augmented by a hierarchy of truth predicates. We establish a close correspondence between these theories. 

Let $\alpha$ be an ordinal and let $\PA(\set{\Tr_\beta: \beta<\alpha})$ 
be Peano arithmetic extended by axioms which express, for each $\beta<\alpha$, that 
$\Tr_\beta(x)$ is a truth predicate for the language with predicates $\Tr_\gamma$, for $\gamma<\beta$
and with full induction in the extended language.
The axioms for being a truth predicate for a language $L_\beta$ are the usual Tarski condition for 
compositional definitions of the truth values for connectives and quantifiers. They may be presented as
follows.  Let $L_n$ be a language with truth predicates $\Tr_{0}, \dots, \Tr_{n-1}$. Then, for each $n\in\bN$ we put in
$\PA(\set{\Tr_\beta: \beta<\omega})$ 
\begin{itemize}
\item 
$\forall (t=t')\in L_n (\Tr_n(t=t') \equiv {\rm val}(t) = {\rm val}(t'))$,
\item 
$\forall (t\leq t')\in L_n (\Tr_n(t\leq t') \equiv {\rm val}(t) \leq  {\rm val}(t'))$,
\item for all $i<n$ we have
$\forall x (\Tr_n(\Tr_i(\bar{x})) \equiv \Tr_i(x))$,
\item
$\forall \vp\in L_n (\Tr_n(\neg\vp) \equiv \neg \Tr_n(\vp))$,
\item 
$\forall \vp,\psi \in L_n (\Tr_n(\vp\imp \psi ) \equiv  (\Tr_n(\vp)\imp \Tr_n(\psi)))$,
\item 
$\forall \vp (\Tr_n(\exists x \vp(x)) \equiv \exists x\,\Tr_n(\vp(\bar{x})))$,
\end{itemize}
where $\bar{x}$ is the $x$-th numeral which is a name for an element $x$ in the model, $t$ and $t'$ are
closed terms and ${\rm val}$ is an arithmetical function which computes a value of a closed term.
The laws for other propositional connectives and for the universal quantifier may be easily proved
from the above axioms. For more on theories with truth predicates, also called satisfaction classes we refer to
\cite{KR:90a} and \cite{KR:90b}. 
\begin{theorem}
The following theories are equivalent over the language of Peano arithmetic:
\begin{enumerate}
\item
$\RCA_0 + \RT(!\omega)$,
\item
$\PA(\{\Tr_i: i\in\Nat\})$.
\end{enumerate}
\end{theorem}
\begin{proof}
For the direction from $1.$ to $2.$, we use the fact that the truth for arithmetical formulas with second order
 parameters say $P_0, \dots, P_n$ is many--one reducible
to the the $\omega$-jump of $P_1\oplus\cdots \oplus P_n = \set{(i,j)\in\bN^2 \colon j\in P_i}$.
Now, if $d$ is a proof in $\PA(\{\Tr_i: i\in\Nat\})$
then it uses only finitely many truth predicates, say $T_1,\dots, T_n$. We can define them using
$0^{(\omega)},\dots, 0^{(\omega n)}$ and carry out the proof in $\RCA_0 + \RT(!\omega)$ proving 
the axioms for truth theory of these $\Tr_1,\dots, \Tr_n$.

For the other direction, if $M\models \PA(\{\Tr_i: i\in\Nat\})$ then we can extend
$M$ to a model of $\RCA_0 + \RT(!\omega)$ without changing its first-order part. 
We simply construct a sequence of models 
$M_i$, for $i\in\Nat$ as follows. As $M_0$ we take just $M$ and for $M_{i+1}$ we take all 
sets which are $\Delta^0_1$-definable from the language with truth predicates $\Tr_0,\dots, \Tr_i$.
The sum of all $M_i$ is obviously closed on $\omega$ jumps since it is closed on taking arithmetical truth
for each sequence of second order parameters $P_0,\dots, P_n$.
It follows that such obtained model satisfies $\RCA_0 + \RT(!\omega)$ and since the first-order part 
of both models is the same we get conservativity in the language of $\PA$.
\end{proof}

In \cite{KR:90b} the authors characterize the arithmetical strength of Peano arithmetic with one 
predicate axiomatized as a truth predicate and with induction for the full language. 
Let $\alpha$ be an ordinal. We define
 $\omega_0(\alpha)=\alpha$ and let $\omega_{k+1}(\alpha)=\omega^{\omega_k(\alpha)}$.
Now, for an ordinal $\alpha$ let $\varepsilon_\alpha$ be the $\alpha$th ordinal $\beta$ with 
the property $\omega^\beta=\beta$. Thus, \emph{the first} such ordinal, $\varepsilon_0$, 
is the limit of $\omega_k(0)$  and $\varepsilon_{\alpha+1}$ is the limit 
of $\omega_k(\varepsilon_\alpha +1)$, where $k\in\Nat$. For limit $\lambda$, one may prove that
$\varepsilon_{\lambda}$ is the limit of $\varepsilon_{\lambda_k}$, where $\lambda_k$ is 
a sequence of ordinals converging to $\lambda$. Of course, in order to define such ordinals 
in arithmetic one needs to define also a coding system which would represent such ordinals
as natural numbers. After representing the ordering up to $\alpha$ in arithmetic, one can 
define the principle of transfinite induction up to $\alpha$, $\TI(\alpha)$.

In \cite{KR:90b} the following theorem is proved.

\begin{theorem}[\cite{KR:90b}]
The arithmetical consequences of $\PA(\Tr_0)$ are exactly the consequences of the theory 
$\PA+ \{ \TI(\varepsilon_{\omega_k(0)})\colon k\in\Nat\}$.
\end{theorem}

Our results allows us to characterize the arithmetical strength of Peano arithmetic with $\omega$ many truth predicates.
Let us define a sequence $\alpha_0=\varepsilon_0$ and $\alpha_{k+1}=\varepsilon_{\alpha_k}$, for $k\in \Nat$.
The limit of this sequence is usually denoted by $\varphi_2(0)$ in the Veblen notation system for ordinals.
The proof theoretic ordinal of the theory $\ACA^+_0$ is $\varphi_2(0)$ (see \cite{Rat:91} for a proof). The arithmetical equivalence
of this theory with $\PA(\{\Tr_i\colon i\in\Nat\})$ allows us to characterize the latter theory by transfinite induction.

\begin{theorem}
The arithmetical consequences of $\PA(\{\Tr_i\colon i\in\Nat\})$ are exactly the consequences of the theory 
$\PA+ \{ \TI(\alpha)\colon \alpha < \varphi_2(0) \}$.
\end{theorem}

\section{Conclusion and Future Research}
We have characterized the effective and the proof-theoretical content of a natural combinatorial Ramsey-type 
theorem due to Pudl\`ak and R\"odl \cite{Pud-Rod:82} and, independently, to Farmaki \cite{Far-Neg:08}.
We have proved that the theorem has computable instances all of whose solutions compute $0^{(\omega)}$, the Turing
degree of arithmetic truth. Moreover, we have shown that the theorem exactly 
captures closure under $\omega$-jump over Computable Mathematics. The theorem 
is interestingly related to Banach space theory because of its equivalent
formulation in terms of Schreier families. 

We now indicate two natural directions for future work on the subject. 

First, we conjecture that our results generalize to the transfinite 
generalizations above $\omega$ of the notions of large set, Schreier family, 
and Turing jump. The notions of $\alpha$-large set, $\alpha$-Schreier family, 
and $\alpha$-Turing jump are all well-defined and studied 
for every countable ordinal (see, respectively, \cite{Ket-Sol:81}, 
\cite{Far-Neg:08}, and \cite{Rog} for definitions). As mentioned
in the introduction, $\RT(!\omega)$ generalizes nicely to colorings of 
$\alpha$-Schreier families, or, equivalently, of exactly $\alpha$-large sets. 
We conjecture that a modification of our arguments will show that, for 
each fixed $\alpha$, the principle $\RT(!\alpha)$ generalizing $\RT(!\omega)$
to colorings of $!\alpha$-large sets is equivalent --- over Computable Mathematics
--- to the closure under the $\alpha$-th Turing jump. Thus, the full theorem 
$\forall \alpha \RT(!\alpha)$ would be equivalent to the system $\ATR_0$ 
(Arithmetical Transfinite Recursion, see \cite{Sim}). Provability in $\ATR_0$ 
can be easily proved by inspection of the proof by Pudl\`ak and R\"odl \cite{Pud-Rod:82}
(using Nash-Williams Theorem) or else by using the $\Sigma^0_1$-Ramsey Theorem.

A second direction for future work is the following. Since $\RT(!\omega)$ is
at least as strong as Ramsey's Theorem it is obviously possible to obtain 
finite independence results for Peano Arithmetic by imposing a suitable largeness
condition (see \cite{DeS:11} for a concrete example). 
A corollary of our results is that $\RT(!\omega)$ implies 
over $\RCA_0$ the well-ordering of the proof-theoretic ordinal of the system $\ACA_0^+$, 
i.e., $\varphi_2(0)$ in Veblen notation.
Using (as of now standard) techniques of miniaturization it is then possible to extract 
from $\RT(!\omega)$ finite first-order independence results in the 
spirit of the Paris-Harrington principle~\cite{Par-Har:77} but for the {\it much} stronger
principle $\ACA_0^+$. 
The hope that finite independence results for systems {\it stronger\/} than Peano Arithmetic 
could be extracted from $(\forall \alpha)\RT(!\alpha)$ is expressed in \cite{Far-Neg:08}. Our results
for $\RT(!\omega)$ confirm this expectation already for $\alpha=\omega$. Details will be reported elsewhere.

\end{document}